\documentclass[10pt, a4paper]{amsart}
\usepackage[dvipsnames]{xcolor}
\usepackage{amsmath, amssymb,amsthm}
\usepackage[extra]{tipa}
\usepackage{lettrine}
\usepackage[Sonny]{fncychap}
\usepackage[english]{babel}
\usepackage{amsmath, amssymb,amsthm}
\usepackage[all]{xy}
\usepackage{tabularx}
\usepackage[applemac]{inputenc}
\usepackage{graphicx} 
\usepackage[dvipsnames]{xcolor}
\usepackage{hyperref}
\usepackage[applemac]{inputenc}
\usepackage{enumitem}
\usepackage{ bbold }
\usepackage{geometry}
\usepackage{prerex}
\usetikzlibrary{fit}
\usepackage{pdflscape}
\usepackage{mathtools}
\usepackage{tikz-cd}

\newcommand\smvee{\raise0.9ex\hbox{$\scriptscriptstyle\vee$}}
\makeatletter
\def\oversortoftilde#1{\mathop{\vbox{\m@th\ialign{##\crcr\noalign{\kern3\p@}%
      \sortoftildefill\crcr\noalign{\kern3\p@\nointerlineskip}%
      $\hfil\displaystyle{#1}\hfil$\crcr}}}\limits}

\def\sortoftildefill{$\m@th \setbox\z@\hbox{$\braceld$}%
  \braceld\leaders\vrule \@height\ht\z@ \@depth\z@\hfill\braceru$}

\makeatother

\usepackage{mathrsfs}

    \DeclareFontFamily{U}{wncy}{}
    \DeclareFontShape{U}{wncy}{m}{n}{<->wncyr10}{}
    \DeclareSymbolFont{mcy}{U}{wncy}{m}{n}
    \DeclareMathSymbol{\Sh}{\mathord}{mcy}{"58}

\DeclareMathOperator{\Norm}{Norm}
\DeclareMathOperator{\Ker}{Ker}

\DeclareMathOperator{\Q}{\mathbf{Q}}
\DeclareMathOperator{\Z}{\mathbf{Z}}
\DeclareMathOperator{\C}{\mathbf{C}}
\DeclareMathOperator{\Gal}{Gal}

\DeclareMathOperator{\Spec}{Spec}

\DeclareMathOperator{\SL}{SL}

\DeclareMathOperator{\rk}{rk}

\newcommand{\cf}{\textit{cf. }}
\newcommand{\ie}{\textit{i.e. }}

\theoremstyle{definition}

\newtheorem{rems}{Remarks}[section]

\theoremstyle{plain}
\newtheorem{thm}{Theorem}[section]

\newtheorem{lem}[thm]{Lemma}

\newtheorem{prop}[thm]{Proposition}

\title{Level compatibility in Sharifi's conjecture}
\author{Emmanuel Lecouturier and Jun Wang}

\begin{document}
\maketitle

\begin{abstract} 
Sharifi has constructed a map from the first homology of the modular curve $X_1(M)$ to the $K$-group $K_2(\Z[\zeta_M, \frac{1}{M}])$, where $\zeta_M$ is a primitive $M$th root of unity. We study how these maps relate when $M$ varies. Our method relies on the techniques developed by Sharifi and Venkatesh. 
\end{abstract}
\tableofcontents

\section{Introduction and notation}\label{Section_introduction}

Sharifi \cite{Sharifi_annals} has constructed a beautiful and explicit map between modular symbols and a cyclotomic $K$-group. This map is conjecturally annihilated by a certain Eisenstein ideal. This conjecture, despite its apparent simplicity, turns out to be highly non-trivial and has lead to much work in recent years, in particular by Fukaya--Kato \cite{Fukaya_Kato} and more recently by Sharifi--Venkatesh \cite{Sharifi_Venkatesh}. This paper is devoted to the study of certain norm relations satisfied by Sharifi's map. This aspect has been studied before by Fukaya--Kato and Scott \cite{Scott}. Their results are however quite restrictive. We use the techniques developed by Sharifi and Venkatesh to remove most of these restrictions. Our main motivation is to apply the results of the present note to obtain results toward the Birch and Swinnerton--Dyer conjecture in the ``Eisenstein'' case \cite{BSD_Eisenstein}. We now set up some notation and describe our results in details.

Fix an algebraic closure $\overline{\Q}$ of $\Q$. For any integer $M \geq 4$, choose a primitive $M$th root of unity $\zeta_M \in \overline{\Q}$ such that for all $M' \mid M$, we have $\zeta_{M'} = \zeta_M^{M/M'}$. Let $\Gamma_1(M) = \{\begin{pmatrix} a & b \\ c & d \end{pmatrix} \in \SL_2(\Z) \text{ such that } a-1 \equiv c \equiv 0 \text{ (modulo } M\text{)}\}$ and denote by $X_1(M)$ the compact modular curve (over $\C$) of level $\Gamma_1(M)$. Let $C_M = \Gamma_1(M) \backslash \mathbf{P}^1(\Q)$ be the set of cusps of $X_1(M)$, and $C_M^0$ be those cusps in $C_M$ of the form $\Gamma_1(M) \cdot \frac{a}{b}$ with $\gcd(a,b)=1$ and $a \not\equiv 0  \text{ (modulo } M \text{)}$ (in the case $b=0$ we have the cusp $\Gamma_1(M)\cdot \infty$). Let $H_1(X_1(M), C_M, \Z)$ be the singular homology of $X_1(M)$ relative to $C_M$. If $\alpha$ and $\beta$ are in $\mathbb{P}^1(\Q)$, let $\{\alpha, \beta\}$ be the class in $H_1(X_1(M), C_M, \Z)$ of the hyperbolic geodesic from $\alpha$ to $\beta$ in $X_1(M)$.

Let $\xi_M : \Z[\Gamma_1(M) \backslash \SL_2(\Z)] \rightarrow H_1(X_1(M), C_M, \Z)$ be the (modified) Manin map: it sends a coset $\Gamma_1(M) \cdot \begin{pmatrix} a & b \\ c & d \end{pmatrix}$ to $\{-\frac{d}{Mb}, -\frac{c}{Ma}\}$ (it is the usual Manin map composed with the Atkin--Lehner involution $W_M$). Manin showed that $\xi_M$ is surjective. Let $S_M^0 \subset \Gamma_1(M) \backslash \SL_2(\Z)$ be the subset consisting of $\Gamma_1(M) \cdot \begin{pmatrix} a & b \\ c & d \end{pmatrix}$ with $M \nmid c$ and $M \nmid d$. The restriction $\xi_M^0 : \Z[S_M^0] \rightarrow H_1(X_1(M), C_M, \Z)$ is surjective (\cf \cite[\S 2.1.3]{Sharifi_survey}).

If $A$ is a commutative ring, let $K_2(A)$ be the second K-group of $A$, as defined by Quillen. For any $x,y \in A^{\times}$, there is an element $\{ x, y \}$ of $K_2(A)$, called the \emph{Steinberg symbol} of $x$ and $y$. It is bilinear in $x$ and $y$ and has the property that if $x+y=1$ then $\{ x, y \} = 1$.

There is an action of $\Gal(\Q(\zeta_M)/\Q)$ (and in particular of the complex conjugation) on $K_2(\Z[\zeta_M, \frac{1}{M}])$. We denote by $\mathcal{K}_M$ the largest quotient of $K_2(\Z[\zeta_M, \frac{1}{M}]) \otimes \Z[\frac{1}{2}]$ on which the complex conjugation acts trivially. The map $\Z[S_M^0] \rightarrow \mathcal{K}_M$ sending $\Gamma_1(M)  \cdot \begin{pmatrix} a & b \\ c & d \end{pmatrix}$ to the Steinberg symbol $\{ 1 - \zeta_M^c, 1-\zeta_M^d \}$ factors through $\xi_M^0$ (\cf \cite[\S 2.1.4]{Sharifi_survey}), and thus induces a map
$$\varpi_M : H_1(X_1(M), C_M^0, \Z) \rightarrow \mathcal{K}_M \text{ .}$$
Sharifi conjectured that $\varpi_M$ is annihilated by the Hecke operators $T_{\ell}-\ell\langle \ell \rangle - 1$ for primes $\ell$ not dividing $M$ (\cf the remark after Theorem 4.3.6 in \cite{Sharifi_Venkatesh}). This conjecture has a history of partial results: \cite{Fukaya_Kato}, \cite{LW} and most recently \cite{Sharifi_Venkatesh}. In particular, the restriction of $\varpi_M$ to $H_1(X_1(M), \Z)$ is known to be annihilated by $T_{\ell}-\ell\langle \ell \rangle - 1$ for primes $\ell$ not dividing $M$ (\cf \cite[Theorem 4.3.7]{Sharifi_Venkatesh}, where we warn the reader that they use \emph{usual} Manin symbols and \emph{dual} Hecke operators).

Another important aspect of Sharifi's theory is how the maps $\varpi_M$ relate with each others when varying $M$. This has been studied under some assumptions in \cite{Fukaya_Kato} and \cite{Scott}. If $p$ is a prime, there are two degeneracy maps $\pi_1, \pi_2 : X_1(Mp) \rightarrow X_1(M)$ given on the upper-half plane by $\pi_1 : z\mapsto z$ and $\pi_2 : z \mapsto pz$. On the $K$-side, there is a norm map $\Norm : \mathcal{K}_{Mp} \rightarrow \mathcal{K}_M$. Our main main result is the following.

\begin{thm}\label{main_thm}
Let $p\geq 2$ be a prime number and $M \geq 4$. Let $C \subset C_{Mp}^0$ be a subset of cusps which are all in the same orbit under the action of $\Ker((\Z/Mp\Z)^{\times} \rightarrow (\Z/M\Z)^{\times})$ (the action being given by diamond operators). 

\begin{enumerate}
\item\label{main_thm_i} Assume that $p$ divides $M$. We have a commutative diagram
\begin{center}
\begin{tikzcd}
H_1(X_1(Mp), C, \Z) \arrow[r, "\varpi_{Mp}"] \arrow[d, "\pi_1"] & \mathcal{K}_{Mp} \arrow[d, "\Norm"] \\
H_1(X_1(M), \Z)\arrow[r, "\varpi_{M}"] &  \mathcal{K}_{M}
\end{tikzcd}
\end{center}

\item\label{main_thm_ii} Assume that $p$ does not divide $M$. We have a commutative diagram
\begin{center}
\begin{tikzcd}
H_1(X_1(Mp), C, \Z) \arrow[r, "\varpi_{Mp}"] \arrow[d, "\pi_1 - \langle p \rangle \pi_2"] & \mathcal{K}_{Mp} \arrow[d, "\Norm"] \\
H_1(X_1(M), \Z)\arrow[r, "\varpi_{M}"] &  \mathcal{K}_{M}
\end{tikzcd}
\end{center}
Here, $\langle p \rangle$ is the $p$th diamond operator, induced by the action of a matrix $\begin{pmatrix} a & b \\ c & d \end{pmatrix} \in \Gamma_0(M)$ with $d \equiv p \text{ (modulo }M\text{)}$ on $X_1(M)$.
\end{enumerate}
\end{thm}

\begin{rems}
\begin{enumerate}
\item Theorem \ref{main_thm} (\ref{main_thm_i}) has been proved by Fukaya--Kato in \cite[Theorem 5.2.3 (2)]{Fukaya_Kato} after tensoring by $\Z_p$. They use $H^2(G_{\Q(\zeta_M)}, \Z_p(2))$ instead of $\mathcal{K}_M \otimes \Z_p$ (which is isomorphic by the \'etale Chern class map \cite{Tate_K2}). They also do not need to restrict to the subset $C$ of $C_{Mp}^0$. They techniques rely on $p$-adic Hodge theory.
\item Similarly, Theorem \ref{main_thm} (\ref{main_thm_ii}) has been proved (for the absolute homology) by Scott in \cite[Theorem 7]{Scott} after tensoring by $\Z_{\ell}$ for a prime $\ell \neq p$ dividing $M$ (Scott's $p$ is our $\ell$ and vice versa). Scott relies on the techniques of Fukaya--Kato. 
\item Thus, the main novelty of our result is to work with $\Z$ coefficients. This is because we rely instead of the motivic techniques of Sharifi and Venkatesh. 
\item It would be interesting to allow less restrictive conditions on $C$, and in particular replace $H_1(X_1(M), \Z)$ in the bottom line of our diagrams by a relative homology group. We were actually able to improve slightly our result when $C$ contains the cusp $\infty$: \cf diagrams (\ref{diagram_C_0_Mp_bis}) and (\ref{diagram_C_0_Mp_bis_2}). We were not able to go beyond these results because the techniques of Sharifi and Venkatesh essentially deal with the absolute homology of modular curves.
\item The techniques of Sharifi and Venkatesh, combined with the result of \S \ref{section_relative_homology} actually show that the restriction of $\varpi_M$ to $H_1(X_1(M), C_{\infty}, \Z)$ is annihilated by $T_{\ell}-\ell\langle \ell \rangle - 1$ for primes $\ell$ not dividing $M$, where $C_{\infty}$ are the cusps of $X_1(M)$ in the same diamond orbit as $\infty$. This is a slight improvement on the result of Sharifi--Venkatesh (which holds for the restriction of $\varpi_M$ to $H_1(X_1(M), \Z)$).
\end{enumerate}
\end{rems}

By combining Theorem \ref{main_thm} and the results of Fukaya--Kato, one gets the following result.

\begin{thm}\label{thm_Atkin_op}
Let $M \geq 4$. The map $H_1(X_1(M), \Z[\frac{1}{6}]) \rightarrow \mathcal{K}_M \otimes \Z[\frac{1}{6}]$ obtained by restricting $\varpi_M$ to $H_1(X_1(M), \Z)$ and inverting $6$ is annihilated by the Hecke operator $U_{\ell}-1$ for all primes $\ell$ dividing $M$. Here, $U_{\ell}$ is the classical Hecke operator of index $\ell$, corresponding to the double coset of $\begin{pmatrix} 1 & 0 \\ 0 & \ell \end{pmatrix}$.
\end{thm}

\begin{rems}
\begin{enumerate}
\item Sharifi and Venkatesh were unable to prove this result using their techniques. Our result thus completes the proof of Sharifi's conjecture \cite[Conjecture 4.3.5 a.]{Sharifi_Venkatesh} for the absolute homology, after inverting $6$. 
\item Fukaya--Kato proved Theorem \ref{thm_Atkin_op} after tensoring with $\Z_p$ for a prime $p\geq 5$ dividing $M$. Our trick is to use Fukaya--Kato's result after adding $p$ to the level, and then descend using Theorem \ref{main_thm} (\ref{main_thm_ii}). The reason we have to invert $6$ is that Fukaya and Kato assume $p \nmid 6$ (note that $2$ is inverted anyway in the definition of $\varpi_M$). It would be nice to be able to avoid inverting $3$ in our result.
\end{enumerate}
\end{rems}

The plan of this paper is as follows. In section \ref{section_preliminaries}, we recall some constructions of Sharifi and Venkatesh. In section \ref{section_relative_homology}, we explain how to use the cocycle of Sharifi--Venkatesh to produce a map on a certain relative homology group. Finally, in section \ref{section_proof}, we prove Theorems \ref{main_thm} and \ref{thm_Atkin_op}.

\section{Reminders from the work of Sharifi--Venkatesh}\label{section_preliminaries}
Sharifi and Venkatesh constructed a cocycle $$\Theta : \SL_2(\Z) \rightarrow K_2/\Z\cdot \{ -z_1, -z_2 \} \text{ .}$$
Here, $K_2 := K_2(\Q(\mathbf{G}_m^2)) = K_2(\Q(z_1, z_2))$ carries a left action of $\SL_2(\Z)$ induced by the natural right action of $\SL_2(\Z)$ on $\mathbf{G}_m^2$ given by $(z_1, z_2) \cdot \begin{pmatrix} a & b \\ c & d \end{pmatrix} = (z_1^az_2^c, z_1^bz_2^d)$. Furthermore, $\Z\cdot \{ -z_1, -z_2 \}$ is the subgroup of $K_2$ generated by the Steinberg symbol of $-z_1$ and $-z_2$. The cocycle $\Theta$ actually takes values in $K_2^{(0)}/ \{ -z_1, -z_2 \}$, where $K_2^{(0)}$ is the subgroup of $K_2$ fixed by the pushforward  $[m]_*$ of the multiplication by $m$ map for all $m \in \mathbf{N}$ (\cf \cite[\S 4.1.2]{Sharifi_Venkatesh}). 

Let us recall a characterization of $\Theta$. Let $K_1 = \bigoplus_D K_1(\Q(D)) = \bigoplus_D \Q(D)^{\times} $ where $D$ runs through all the irreducible divisors of $\mathbf{G}_m^2$. There is a divisor map $\partial : K_2 \rightarrow K_1$ sending a Steinberg symbol $\{ f, g\}$ to the element of $K_1$ whose component in $D$ is $(-1)^{v(f)v(g)}g^{v(f)}f^{-v(g)}$, where $v$ is the valuation coming from $D$ (\cf \cite[(2.6)]{Sharifi_Venkatesh}). The map $\partial$ induces an embedding $\partial : K_2^{(0)}/\Z\cdot \{ -z_1, -z_2 \} \hookrightarrow K_1$ (\cf \cite[\S 3.2 and Lemma 4.1.2]{Sharifi_Venkatesh}). As in \cite[\S 3.2]{Sharifi_Venkatesh}, for any $a, c \in \Z$ with $\gcd(a,c)=1$ there is a special element $\langle a , c \rangle$ in $K_1$ which is supported on the divisor $D : 1-z_1^az_2^c=0$ and is given there by the function $1-z_1^bz_2^d$ for any $b, d\in \Z$ such that $ad-bc=1$ (this is independent of the choice of $b$ and $d$). For any $\gamma = \begin{pmatrix} a & b \\ c & d \end{pmatrix} \in \SL_2(\Z)$, $\Theta(\gamma)$ is characterized by the equality $\partial(\Theta(\gamma)) = \langle b, d \rangle - \langle 0, 1 \rangle$ in $K_1$ (\cf \cite[Proposition 3.3.1]{Sharifi_Venkatesh}). As in the proof of \cite[Proposition 3.3.4]{Sharifi_Venkatesh}, one sees that $\Theta(\gamma)=0$ if $\gamma(0)=0$, \ie if $\gamma = \begin{pmatrix} 1 & 0 \\ m & 1 \end{pmatrix}$ for some $m \in \Z$.

Finally, let us recall how Sharifi and Venkatesh (\cf \cite[\S 4.2.1]{Sharifi_Venkatesh}) specialize $\Theta$ to a cocycle $\Theta_M : \Gamma_0(M) \rightarrow K_2(\Z[\zeta_M, \frac{1}{M}])/\Z\cdot \{-1, -\zeta_M\}$ (for every $M \geq 4$). Here, the action of $ \Gamma_0(M)$ on $K_2(\Z[\zeta_M, \frac{1}{M}])$ is given as follows: we have a surjective group homomorphism $\Gamma_0(M) \rightarrow (\Z/M\Z)^{\times}$ sending $\begin{pmatrix} a & b \\ c & d \end{pmatrix}$ to $d \text{ (modulo }M\text{)}$, and $(\Z/M\Z)^{\times}$ acts on $K_2(\Z[\zeta_M, \frac{1}{M}])$ via the usual identification $(\Z/M\Z)^{\times} \simeq \Gal(\Q(\zeta_M)/\Q)$. The idea is to evaluate $\Theta(\gamma)$ at $(z_1, z_2) = (1, \zeta_M)$. One cannot in general evaluate an element of $K_2$ at $(1, \zeta_M)$. But if $\gamma \in \Gamma_0(M)$, we have $\Theta(\gamma) \in H^2(U_{\gamma}, 2)/\{-z_1, -z_2\}$ where $U_{\gamma}$ is the open subset of $\mathbf{G}_m^2$ which is the complement of $\{z_1^bz_2^d = 1\} \cup \{z_2=1\}$ (we are using motivic cohomology, \cf \cite[\S 2.1]{Sharifi_Venkatesh}). Since $(1, \zeta_M) \in U_{\gamma}$, there is a specialization map $s_M^* : H^2(U_{\gamma}, 2) \rightarrow H^2(\Q(\zeta_M), 2) \simeq K_2(\Q(\zeta_M))$. By \cite[Corollary 4.2.5]{Sharifi_Venkatesh}, $\Theta_M(\gamma) := s_M^*(\Theta(\gamma))$ actually belongs to the subgroup $K_2(\Z[\zeta_M, \frac{1}{M}])/\Z\cdot \{-1, -\zeta_M\}$ of $K_2(\Q(\zeta_M))/\Z\cdot \{-1, -\zeta_M\}$.

\section{From cocycles to relative homology}\label{section_relative_homology}

In this section, we explain how the cocycle $\Theta_M$ of section \ref{section_preliminaries} gives rise to a group homomorphism $$\tilde{\Theta}_M : H_1(X_1(M), C_0, \Z) \rightarrow K_2(\Z[\zeta_M, \frac{1}{M}])/\Z\cdot \{-1, -\zeta_M\}$$ where $C_0$ is the set of cusps of $X_1(M)$ which are in the same diamond orbit as the cusp $\Gamma_1(M) \cdot 0$. There is a map $f : \Gamma_0(M) \rightarrow  H_1(X_1(M), C_0, \Z)$ given by $\gamma \mapsto \{0, \gamma 0\}$.
 Note that $f$ is a $1$-cocycle (for the action of $\Gamma_0(M)$ on $H_1(X_1(M), C_0, \Z)$ via diamond operators). 
\begin{prop}\label{prop_cocycle_homology}
Let $T$ be a $\Z[(\Z/M\Z)^{\times}/\pm 1]$-module (where $M>3$). Let $u : \Gamma_0(M) \rightarrow T$ be a $1$-cocycle satisfying $u(\gamma)=0$ for any $\gamma \in \Gamma_1(M)$ such that $\gamma c = c$ for some $c \in \mathbf{P}^1(\Q)$. Then $u$ factors through the map $f : \Gamma_0(M) \rightarrow  H_1(X_1(M), C_0, \Z)$, thus inducing a $1$-cocycle $\tilde{u} : H_1(X_1(M), C_0, \Z) \rightarrow T$ (for the action of $(\Z/M\Z)^{\times}/\pm1$ on both sides).
\end{prop}
\begin{proof} For notational simplicity, let $G = (\Z/M\Z)^{\times} / \pm 1$, $\Gamma_0 = \Gamma_0(M)/\pm 1$ and $\Gamma_1 = \Gamma_1(M) \subset \Gamma_0$. If $\gamma \in \Gamma_0$, we let $\langle \gamma \rangle \in G$ be the class of the lower-right corner of $\gamma$.
For any $c \in \mathbf{P}^1(\Q)$, let $\gamma_c \in \Gamma_1$ be a generator of the stabilizer of $c$ in $\Gamma_1$. By assumption, we have $u(\gamma_c) = 0$ for all $c \in \mathbf{P}^1(\Q)$. Thus, $u$ induces a group homomorphism $u' : \Z[G \times \Gamma_0]/I \rightarrow T$ given by $u'(g, \gamma) = g\cdot u(\gamma)$, where $I$ is subgroup of $ \Z[G \times \Gamma_0]$ generated by the elements $(g, \gamma\gamma') - (g, \gamma) - (g\langle \gamma \rangle, \gamma')$ and by the $(1, \gamma_c) - (1, 1)$ for all $g \in G$, $\gamma, \gamma' \in \Gamma_0(M)$ and $c \in \mathbf{P}^1(\Q)$. 

It suffices to prove that the map $ \Z[G \times \Gamma_0]/I \rightarrow H_1(X_1(M), C_0, \Z)$ sending $(g, \gamma)$ to $g\cdot \{0, \gamma 0 \}$ is an isomorphism. It is clearly surjective, so it is enough to show that $ \Z[G \times \Gamma_0]/I$ is a free $\Z$ module of same rank as $H_1(X_1(M), C_0, \Z)$. Since $M>3$, the group $\Gamma_1$ is torsion-free and we have $H_1(\Gamma_1, \Z) \simeq H_1(Y_1(M), \Z)$. By Shapiro's lemma for group homology, we have $H_1(\Gamma_1, \Z) \simeq H_1(\Gamma_0, \Z[G])$. By definition of group homology, we have an exact sequence
$$0 \rightarrow H_1(\Gamma_0, \Z[G]) \rightarrow \Z[G \times \Gamma_0]/J \xrightarrow{\partial} \Z[G] \rightarrow \Z \rightarrow 0$$
where $\partial (g, \gamma) = g\cdot \langle \gamma \rangle - g$ and $J$ is the subgroup of $\Z[G \times \Gamma_0]$ generated by $(g, \gamma\gamma') - (g, \gamma) - (g\langle \gamma \rangle, \gamma')$ for $g \in G$, $\gamma, \gamma' \in \Gamma_0(M)$. The last map $\Z[G] \rightarrow \Z$ is the augmentation (degree) map (note that $J$ is indeed in the kernel of $\partial$).

We also have a short exact sequence 
$$0 \rightarrow \Z \rightarrow \Z[C_M] \rightarrow H_1(Y_1(M), \Z) \rightarrow H_1(X_1(M), \Z) \rightarrow 0$$
where $C_M$ is the set of cusps of $Y_1(M)$. Here, the map $ \Z[C_M]$ sends a cusp $c$ to the homology class of a small loop around $c$ in $Y_1(M)$. Under the isomorphism $H_1(Y_1(M), \Z) \simeq H_1(\Gamma_0, \Z[G])$ and the embedding $H_1(\Gamma_0, \Z[G]) \hookrightarrow \Z[G \times \Gamma_0]/J$ described above, the map $\Z[C_M] \rightarrow H_1(\Gamma_0, \Z[G])$ sends a cusp $c$ to the class of $\gamma_c - 1$ in $\Z[G \times \Gamma_0]/J$. Therefore, we have an exact sequence
$$0 \rightarrow H_1(X_1(M), \Z) \rightarrow \Z[G \times \Gamma_0]/I \xrightarrow{\partial} \Z[G] \rightarrow \Z \rightarrow 0 \text{ .}$$
This shows that $ \Z[G \times \Gamma_0]/I$ is a free $\Z$-module of rank $\rk_{\Z} H_1(X_1(M), \Z) + \#G -1$. We have $\# G = \# C_0$, and the exact sequence
$$0 \rightarrow H_1(X_1(M), \Z) \rightarrow H_1(X_1(M),  C_0, \Z) \rightarrow \Z[C_0] \rightarrow \Z \rightarrow 0$$
shows that $\rk_{\Z}  \Z[G \times \Gamma_0]/I  = \rk_{\Z} H_1(X_1(M),  C_0, \Z)$, as wanted.

\end{proof}

\section{Proofs of the theorems}\label{section_proof}

We start with the following lemma (we thank Venkatesh for explaining this to us). 

\begin{lem}\label{key_lem} Let $M \geq 4$ and $p\geq 2$ be a prime. Let $\alpha = \begin{pmatrix} 1 & 0 \\ 0 & p \end{pmatrix}$.
Let $\phi_p : \Gamma_0(Mp) \rightarrow \Gamma_0(M)$ be the group homomorphism sending $\begin{pmatrix} a & b \\ c & d \end{pmatrix}$ to $\begin{pmatrix} a & pb \\ c/p & d \end{pmatrix}$. We have a commutative diagram
\begin{center}
\begin{tikzcd}
\Gamma_0(Mp) \arrow[r, "\Theta"] \arrow[d, "\phi_p"] & K_2^{(0)}/\Z \cdot \{-z_1, -z_2\} \arrow[d, "\alpha_*"] \\
\Gamma_0(M) \arrow[r, "\Theta"] &  K_2^{(0)}/\Z \cdot \{-z_1, -z_2\}
\end{tikzcd}
\end{center}
where $\alpha_*$ is the trace map induced by $\alpha$.
\end{lem}
\begin{proof}
Since $\partial : K_2^{(0)}/\Z \cdot \{-z_1, -z_2\} \hookrightarrow K_1$ is injective, it suffices to prove that the following diagram is commutative:

\begin{center}
\begin{tikzcd}
\Gamma_0(Mp) \arrow[r, "\partial \circ \Theta"] \arrow[d, "\phi_p"] & K_1 \arrow[d, "\alpha_*"] \\
\Gamma_0(M) \arrow[r, "\partial \circ \Theta"] &  K_1
\end{tikzcd}
\end{center}
In other words, if $\gamma = \begin{pmatrix} a & b \\ c & d \end{pmatrix} \in \Gamma_0(Mp)$ then it suffices to check that 
\begin{equation}\label{equation_lemma}
\alpha_*(\langle b, d \rangle - \langle 0, 1 \rangle) = \langle b/p, d \rangle - \langle 0, 1 \rangle \text{ .}
\end{equation}
As in \cite[(3.2)]{Sharifi_Venkatesh}, we have $\langle b , d \rangle = \gamma^* \langle 0, 1 \rangle$ where $\gamma^* : K_1 \rightarrow K_1$ is the pull-back induced by the right action of $\gamma$ on $\mathbf{G}_m^2$. Note that we have $\gamma^* = (\gamma^{-1})_*$. Thus, we have:
$$\alpha_* \langle b, d \rangle = \alpha_* (\gamma^{-1})_* \langle 0, 1 \rangle = (\gamma^{-1}\cdot \alpha)_* \langle 0, 1 \rangle =  (\alpha^{-1} \cdot \gamma^{-1}\cdot\alpha)_* \alpha_* \langle 0, 1 \rangle =(\alpha^{-1} \cdot \gamma \cdot\alpha)^* \alpha_* \langle 0, 1 \rangle \text{ .}$$
One easily checks that $\alpha_*\langle 0 , 1 \rangle = \langle 0, 1 \rangle$ (this is essentially the same computation as \cite[(2.8)]{Sharifi_Venkatesh}). Since $\alpha^{-1} \cdot \gamma \cdot\alpha = \begin{pmatrix} a & pb \\ c/p & d \end{pmatrix}$, we get $\alpha_* \langle b, d \rangle  = \langle c/p, d \rangle$. We also get $\alpha_* \langle 0, 1 \rangle = \langle 0, 1 \rangle$. This proves (\ref{equation_lemma}).
\end{proof}

We are now ready to prove Theorem \ref{main_thm}. Let $\gamma' \in \Gamma_0(Mp)$ and let $\gamma = \phi_p(\gamma') \in \Gamma_0(M)$ be as in Lemma \ref{key_lem}. Let $f : \mathbf{G}_m^2 \rightarrow \mathbf{G}_m^2$ sending $(z_1, z_2)$ to $(z_1, z_2^p)$ (note that $f$ is induced by the action of the matrix $\alpha$ of Lemma \ref{key_lem}). Let $U = U_{\gamma}$ be as in section \ref{section_preliminaries} and $U' = f^{-1}(U) \cap U_{\gamma'}$. Both $U$ and $U'$ are open subschemes of $\mathbf{G}_m^2$, and we have $(1, \zeta_M) \in U$ and $(1, \zeta\cdot \zeta_{Mp}) \in U'$ for all $p$th root of unity $\zeta$.

Consider the following cartesian diagram of schemes:
\begin{center}
\begin{tikzcd}
X\arrow[r] \arrow[d] & U' \arrow[d, "f"] \\
\Spec(\Q(\zeta_M)) \arrow[r, "s_M"] &  U
\end{tikzcd}
\end{center}
where $s_M$ is given by the closed point $(1, \zeta_M) \in U$, and $X$ makes the diagram cartesian by definition. There are two distinct cases for $X$ depending on whether $p$ divides $M$ or not.

\subsection{The case $p \mid M$}
Assume first that $p$ divides $M$. Then $X \simeq \Spec(\Q(\zeta_{Mp}))$ and the map $X \rightarrow U'$ comes from the point $(1, \zeta_{Mp}) \in U'$. Passing to motivic cohomology (\cf \cite[Lemma 2.1.1]{Sharifi_Venkatesh}), we get a commutative diagram
\begin{center}
\begin{tikzcd}
H^2(U', 2) \arrow[r, "s_{Mp}^*"] \arrow[d, "f_*"] & K_2(\Q(\zeta_{Mp})) \arrow[d, "\Norm"] \\
H^2(U, 2) \arrow[r, "s_M^*"] &  K_2(\Q(\zeta_M))
\end{tikzcd}
\end{center}
and hence a commutative diagram

\begin{center}
\begin{tikzcd}
H^2(U', 2)/\Z \cdot \{-z_1, -z_2\} \arrow[r, "s_{Mp}^*"] \arrow[d, "f_*"] & K_2(\Q(\zeta_{Mp}))/\Z \cdot \{-1, -\zeta_{Mp}\} \arrow[d, "\Norm"] \\
H^2(U, 2)/ \Z \cdot \{-z_1, -z_2\} \arrow[r, "s_M^*"] &  K_2(\Q(\zeta_M))/ \Z \cdot \{-1, -\zeta_M\}
\end{tikzcd}
\end{center}

By Lemma \ref{key_lem}, we get $s_{M}^*(\Theta(\gamma)) = \Norm(s_{Mp}^*(\Theta(\gamma')))$, \ie 
\begin{equation}\label{norm_eq_1}
\Norm(\Theta_{Mp}(\gamma')) = \Theta_M(\gamma) \text{ .}
\end{equation}

By Proposition \ref{prop_cocycle_homology}, equation (\ref{norm_eq_1}) yields the following commutative diagram
\begin{center}
\begin{equation}\label{diagram_C_0_Mp}
\begin{tikzcd}
H_1(X_1(Mp), C_0', \Z) \arrow[r, "\tilde{\Theta}_{Mp}"]\arrow[d, "\pi_2"] & K_2(\Z[\zeta_{Mp}, \frac{1}{M}])/\Z\cdot \{-1, -\zeta_{Mp}\}\arrow[d, "\Norm"] \\
H_1(X_1(M), C_0, \Z) \arrow[r, "\tilde{\Theta}_M"] &   K_2(\Z[\zeta_{M}, \frac{1}{M}])/\Z\cdot \{-1, -\zeta_{M}\}
\end{tikzcd}
\end{equation}
\end{center}
where $C_0'$ (resp. $C_0$) is the set of cusps of $X_1(Mp)$ (resp. $X_1(M)$) in the same orbit as $0$. Since $\{-1, -\zeta_{M}\}$ is acted upon by $-1$ by the complex conjugation, we have canonical projections $ K_2(\Z[\zeta_{M}, \frac{1}{M}])/\Z\cdot \{-1, -\zeta_{M}\} \otimes \Z[\frac{1}{2}] \rightarrow \mathcal{K}_M$ and  $ K_2(\Z[\zeta_{Mp}, \frac{1}{M}])/\Z\cdot \{-1, -\zeta_{Mp}\} \otimes \Z[\frac{1}{2}] \rightarrow \mathcal{K}_{Mp}$. After applying the Atkin--Lehner involution $W_{Mp}$ and $W_M$ to the two lines of (\ref{diagram_C_0_Mp}), we get a commutative diagram

\begin{center}
\begin{equation}\label{diagram_C_0_Mp_bis}
\begin{tikzcd}
H_1(X_1(Mp), C_{\infty}', \Z) \arrow[r, "\varpi_{Mp}"]\arrow[d, "\pi_1"] & \mathcal{K}_{Mp} \arrow[d, "\Norm"]  \\
H_1(X_1(M), C_{\infty}, \Z) \arrow[r, "\varpi_M"] &   \mathcal{K}_M
\end{tikzcd}
\end{equation}
\end{center}
where $C_{\infty}'$ (resp. $C_{\infty}$) is the set of cusps of $X_1(Mp)$ (resp. $X_1(M)$) in the same orbit as $\infty$. 
We have used the facts that $\varpi_{Mp} = \tilde{\Theta}_{Mp} \circ W_{Mp}$ and $\varpi_{M} = \tilde{\Theta}_{M} \circ W_{M}$. This follows from \cite[Proposition 4.3.3]{Sharifi_Venkatesh}, where the authors use usual Manin symbols (whereas our map $\varpi_M$ uses Manin symbols twisted by the Atkin--Lehner involution). 

Now, let $C$ be a subset of cusps of $X_1(Mp)$ as in Theorem \ref{main_thm}. If $C \subset C_{\infty}'$, then Theorem \ref{main_thm} follows from (\ref{diagram_C_0_Mp_bis}) (we just restrict $\varpi_{Mp}$ to $H_1(X_1(Mp), C, \Z) \subset H_1(X_1(Mp), C_{\infty}', \Z)$). Let us explain how to deduce the general case from this special case. Fix $c \in \mathbf{P}^1(\Q)$ such that $\Gamma_1(Mp)\cdot c \in C$. An element of $H_1(X_1(Mp), C, \Z)$ is of the form $\{c, \gamma c\}$ for some $\gamma \in \Gamma_0(Mp)$. The assumption that all the elements of $C$ are in the same diamond orbit under $\Ker((\Z/Mp\Z)^{\times} \rightarrow (\Z/M\Z)^{\times})$ means that we can actually choose $\gamma$ in $\Gamma_0(Mp) \cap \Gamma_1(M)$. 

We have 
\begin{align*}
 \{c, \gamma c\} &= \{c, \infty\} + \{\infty, \gamma \infty\} + \{\gamma \infty, \gamma c\} \\&
 =  \{\infty, \gamma \infty\}  + (\langle \gamma \rangle - 1)\cdot \{\infty, c \} \text{ ,}
 \end{align*}
where $\langle \gamma \rangle$ is the diamond operator associated with $\gamma$. Since $\gamma \in \Gamma_0(Mp) \cap \Gamma_1(M)$, we have $ \pi_2\left( (\langle \gamma \rangle - 1)\cdot \{\infty, c \}\right) = 0$. We also have $\Norm\left( \varpi_{Mp}((\langle \gamma \rangle - 1)\{\infty, c \} ) \right)$ in $ \mathcal{K}_M$ since $\varpi_{Mp}((\langle \gamma \rangle - 1)\{\infty, c \} )$ belongs to $(\langle \gamma \rangle - 1)\cdot \mathcal{K}_{Mp}$. Thus, we have:

\begin{align*}
\Norm(\varpi_{Mp}( \{c, \gamma c\} )) &= \Norm(\varpi_{Mp}( \{\infty, \gamma \infty\} )) \\&
= \varpi_M(\pi_2( \{\infty, \gamma \infty\} )) \\&
=\varpi_M(\pi_2( \{c, \gamma c\} )) \text{ .}
\end{align*}
This concludes the proof of Theorem \ref{main_thm} in the case $p \mid M$.

\subsection{The case $p \nmid M$}
Assume now that $p$ does not divide $M$. Note that in this case we have $(1, \zeta_M) \in U'$. There is an isomorphism $X\simeq \Spec(\Q(\zeta_{Mp})) \sqcup \Spec(\Q(\zeta_{M}))$ such that:
\begin{itemize}
\item the map $X \rightarrow U'$ is given by the two inclusions $(1, \zeta_M) \in U'$ and $(1, \zeta_{Mp}) \in U'$.
\item the map $X \rightarrow \Spec(\Q(\zeta_{M}))$ is  given by the canonical map $\Spec(\Q(\zeta_{Mp})) \rightarrow  \Spec(\Q(\zeta_{M}))$ and the map $\Spec(\Q(\zeta_{M})) \rightarrow  \Spec(\Q(\zeta_{M}))$ induced by the Galois automorphism $\sigma_p : \zeta_M \mapsto \zeta_M^{p}$.
 \end{itemize}
We thus get a commutative diagram

\begin{center}
\begin{tikzcd}
H^2(U', 2)/\Z \cdot \{-z_1, -z_2\} \arrow[r, "s_{Mp}^* \oplus s_M^*"] \arrow[d, "f_*"] & K_2(\Q(\zeta_{Mp}))/\Z \cdot \{-1, -\zeta_{Mp}\} \bigoplus K_2(\Q(\zeta_{M}))/\Z \cdot \{-1, -\zeta_{M}\}  \arrow[d, "\Norm \oplus (\sigma_p)_*"] \\
H^2(U, 2)/ \Z \cdot \{-z_1, -z_2\} \arrow[r, "s_M^*"] &  K_2(\Q(\zeta_M))/ \Z \cdot \{-1, -\zeta_M\}
\end{tikzcd}
\end{center}
In conclusion, if $p\nmid M$ we get $\Norm(\Theta_{Mp}(\gamma')) + (\sigma_p)_*(\Theta_M(\gamma')) = \Theta_M(\gamma)$. Since $(\sigma_p)_* = (\sigma_p^*)^{-1}$, we get
\begin{equation}\label{norm_eq_2}
\Norm(\Theta_{Mp}(\gamma')) = \Theta_M(\gamma) - (\sigma_p^*)^{-1}(\Theta_M(\gamma')) \text{ .}
\end{equation}

As in (\ref{diagram_C_0_Mp_bis}), we get a commutative diagram
\begin{center}
\begin{equation}\label{diagram_C_0_Mp_bis_2}
\begin{tikzcd}
H_1(X_1(Mp), C_{\infty}', \Z) \arrow[r, "\varpi_{Mp}"]\arrow[d, "\pi_1-\langle p \rangle \pi_2"] & \mathcal{K}_{Mp} \arrow[d, "\Norm"]  \\
H_1(X_1(M), C_{\infty}, \Z) \arrow[r, "\varpi_M"] &   \mathcal{K}_M
\end{tikzcd}
\end{equation}
\end{center}
(note that both $\pi_1$ and $\langle p \rangle \pi_2$ send $C_{\infty}'$ to $C_{\infty}$, so that this diagram makes sense). 
An argument identical to the one when $p\mid M$ shows that the diagram of Theorem \ref{main_thm} commutes. This concludes the proof of Theorem \ref{main_thm}.

\subsection{Proof of Theorem \ref{thm_Atkin_op}}
Let us now prove Theorem \ref{thm_Atkin_op}. Let $M \geq 4$ and $p \geq 5$ be a prime. One needs to prove that $\varpi_M \otimes \Z_p : H_1(X_1(M), \Z_p) \rightarrow \mathcal{K}_M \otimes \Z_p$ is annihilated by the operator $U_{\ell}-1$ for any prime $\ell \mid M$. If $p \mid M$, then this is a result of Fukaya--Kato \cite[Theorem 4.3.6]{Sharifi_Venkatesh}. Note that the result of Fukaya--Kato involves dual Hecke operators, but since we use modified Manin symbols our statement involving classical Hecke operators is equivalent. Let us thus assume that $p\nmid M$. By Theorem \ref{main_thm} (\ref{main_thm_ii}), we have a commutative diagram

\begin{center}
\begin{equation}\label{diagram_tensor_Z_p}
\begin{tikzcd}
H_1(X_1(Mp), \Z_p) \arrow[r, "\varpi_{Mp} \otimes \Z_p"]\arrow[d, "\pi_1-\langle p \rangle \pi_2"] & \mathcal{K}_{Mp} \otimes \Z_p \arrow[d, "\Norm"]  \\
H_1(X_1(M), \Z_p) \arrow[r, "\varpi_M \otimes \Z_p"] &   \mathcal{K}_M \otimes \Z_p
\end{tikzcd}
\end{equation}
\end{center}

By the result of Fukaya--Kato, one knows that $\varpi_{Mp} \otimes \Z_p$ is annihilated by the Hecke operator $U_{\ell}-1$. Since $\pi_1-\langle p \rangle \pi_2$ commutes with the action of $U_{\ell}-1$ on both sides, it suffices to prove that $\pi_1-\langle p \rangle \pi_2 : H_1(X_1(Mp), \Z_p) \rightarrow H_1(X_1(M), \Z_p)$ is surjective, which by Nakayama Lemma is equivalent to the surjectivity $\pi_1-\langle p \rangle \pi_2 : H_1(X_1(Mp), \Z/p\Z) \rightarrow H_1(X_1(M), \Z/p\Z)$. By intersection duality, it suffices to prove that $\pi_1^*- \pi_2^* \circ \langle p \rangle^{-1} : H_1(X_1(M), \Z/p\Z) \rightarrow H_1(X_1(Mp), \Z/p\Z)$ is injective. Note that $H_1(X_1(M), \Z/p\Z)$ is canonically isomorphic to the parabolic cohomology $H^1_p(\Gamma_1(M), \Z/p\Z)$, \ie the subgroup of $H^1(\Gamma_1(M), \Z/p\Z)$ consisting of classes of cocycles which are coboundaries when restricted to stabilizers of cusps. Thus, it is enough for us to prove that $\pi_1^*- \pi_2^* \circ \langle p \rangle^{-1} : H^1(\Gamma_1(M), \Z/p\Z) \rightarrow H^1(\Gamma_1(Mp), \Z/p\Z)$ is injective. This follows from \cite[Lemma 1]{Edixhoven}, which says that the map $H^1(\Gamma_1(M), \Z/p\Z)^2 \xrightarrow{\pi_1^*+ \pi_2^*} H^1(\Gamma_1(Mp), \Z/p\Z)$ is injective.

\bibliography{biblio}
\bibliographystyle{plain}
\newpage

\end{document}